\documentclass[oneside,12pt]{amsart}
\usepackage[T2A]{fontenc}
\usepackage[cp1251]{inputenc}
\usepackage{amssymb}
\usepackage{amsxtra}
\usepackage{amsmath}
\usepackage{amsthm}

\usepackage{hyperref}
\usepackage[all]{xy}

\usepackage[russian,english]{babel}

\DeclareMathOperator{\GL}{GL}

\DeclareMathOperator{\Hom}{Hom}

\DeclareMathOperator{\Spec}{Spec}

\DeclareMathOperator{\der}{der}

\DeclareMathOperator{\Gm}{{\mathbf G}_m}
\DeclareMathOperator{\rad}{rad}
\DeclareMathOperator{\corad}{corad}
\DeclareMathOperator{\SL}{SL}

\newcommand{\id}{\text{\rm id}}

\DeclareMathOperator{\ZZ}{{\mathbb Z}}

\newtheorem{lem}{Lemma}[section]
\newtheorem*{thm*}{Theorem}
\newtheorem{thm}[lem]{Theorem}

\newtheorem{cor}[lem]{Corollary}
\newtheorem*{cor*}{Corollary}
\theoremstyle{definition}{  \newtheorem{rem}[lem]{Remark}  }
\theoremstyle{definition}{   }
\theoremstyle{definition}{  \newtheorem{dfn}[lem]{Definition} }

\DeclareMathOperator{\et}{\text{\'et}}

\newcommand{\st}{\scriptstyle}

\newcommand{\Aff}{\mathbb {A}}

\begin{document}

\title{Isotropic reductive groups over discrete Hodge algebras}
\author{Anastasia Stavrova}


\thanks{The author is a grantee of the contest ``Young Russian Mathematics''. The work was supported
by the Russian Science Foundation grant 14-21-00035.
}
\address{Chebyshev Laboratory, St. Petersburg State University,
14th Line V.O. 29B, 199178 Saint Petersburg, Russia}
\email{anastasia.stavrova@gmail.com}

\maketitle

\begin{abstract}
Let $G$ be a reductive group over a commutative ring $R$. We say that $G$ has isotropic rank $\ge n$,
if every normal semisimple reductive $R$-subgroup of $G$ contains $(\Gm_{,R})^n$. We prove that if $G$ has isotropic rank $\ge 1$
and $R$ is a regular domain containing an infinite field $k$, then for
any discrete Hodge algebra $A=R[x_1,\ldots,x_n]/I$ over $R$,
the map $H^1_{Nis}(A,G)\to H^1_{Nis}(R,G)$ induced by evaluation at $x_1=\ldots=x_n=0$, is a bijection.
If $k$ has characteristic $0$, then, moreover,
the map $H^1_{\et}(A,G)\to H^1_{\et}(R,G)$ has trivial kernel. We also prove that
if $k$ is perfect, $G$ is defined over $k$, the isotropic rank of $G$ is $\ge 2$, and $A$ is square-free,
then $K_1^G(A)=K_1^G(R)$, where $K_1^G(R)=G(R)/E(R)$ is the corresponding
non-stable $K_1$-functor, also called the Whitehead group of $G$. The corresponging statements for $G=\GL_n$ were
previously proved by Ton Vorst.
\end{abstract}

\section{Introduction}

Let $R$ be a commutative ring with 1.
A commutative $R$-algebra $A$ is called a discrete Hodge algebra over $R$ if
$A=R[x_1,\ldots,x_n]/I$, where $I$ is an ideal generated by monomials. If $I$ is generated by square-free monomials,
 $A$ is called a square-free discrete Hodge algebra. The simplest example of such an algebra is $R[x,y]/xy$.
Square-free discrete Hodge algebras over a field are also called Stanley--Reisner rings.

Serre's conjecture on modules over polynomial rings, proved by D. Quillen and A. Suslin,
states that any finitely generated projective module over a polynomial ring over a field is free. More generally, the
Bass--Quillen conjecture~\cite{Bass73,Q-serreconj} states that for any regular ring $R$, every finitely generated projective module over
$R[x_1,\ldots,x_n]$ is extended from $R$. In~\cite{Vo-SeHo} T. Vorst proved that, once the Bass--Quillen conjecture holds for
$R$, then it also holds for any discrete Hodge algebra $A$ over $R$, i.e. every finitely generated projective $A$-module
is extended from $R$. Later, S. Mandal~\cite{Man85,Man86} used the same technique to extend
several earlier results on cancellation and extendability for modules
and quadratic spaces over polynomial rings to discrete Hodge algebras.

We generalize the observation of Vorst as follows. Grothendieck topologies are understood in the sense of~\cite[Tag 00ZD]{Stacks}.

\begin{thm}\label{thm:H1}
Let $G$ be a faithfully flat affine group scheme locally of finite presentation over a commutative ring $R$. Let $\tau$
be a Grothendieck topology on $R$-schemes, refined by the fppf topology.

(i) If $H^1_\tau(R[x_1,\ldots,x_n],G)\to H^1_\tau(R,G)$ has trivial kernel (respectively, is bijective)
for any $n\ge 1$, then for any square-free discrete Hodge algebra $A[x_1,\ldots,x_n]/I$ over $R$,
the canonical map $H^1_\tau(A,G)\to H^1_\tau(R,G)$
has trivial kernel (respectively, is bijective).

(ii) Assume in addition that $G$ is smooth. Then the claim of (i) holds for any discrete Hodge algebra $A$ over $R$.
\end{thm}

Let $G$ be a reductive group scheme over $R$ in the sense of~\cite{SGA3}. We say that $G$ has isotropic rank $\ge n$,
if every normal semisimple reductive $R$-subgroup of $G$ contains $(\Gm_{,R})^n$.
Analogs of the Bass--Quillen conjecture for reductive groups $G$ of isotropic rank $\ge 1$ have been established
in many cases, most notably, for tori over regular rings, and for reductive groups over regular
rings containing an infinite field, see~\cite{CTS,CTO,PaStV,AHW} and the references therein.
It is known that, at least for reductive groups over an infinite perfect field,
the isotropy condition is necessary for the Bass--Quillen conjecture to hold~\cite{BaSa17}.

Combining Theorem~\ref{thm:H1} with several of the above results and the infinite field case
of the Serre--Grothendieck conjecture~\cite{PaF}, we obtain the following analogs of the Bass--Quillen
conjecture over discrete Hodge algebras. Note that, formally, the Nisnevich cohomology $H^1_{Nis}(-,G)$ is larger than
the Zariski cohomology, however, the Serre--Grothendieck conjecture implies that they
coincide on regular domains. From this point on, we assume all rings to be Noetherian.

\begin{cor}\label{cor:torsors}
Let $G$ be a reductive group scheme over
a regular domain $R$ containing an infinite field $k$. Assume that $G$ has isotropic rank $\ge 1$.
Then for
any discrete Hodge algebra $A=R[x_1,\ldots,x_n]/I$ over $R$,
the map
$$
H^1_{Nis}(A,G)\to H^1_{Nis}(R,G),
$$
induced by evaluation at $x_1=\ldots=x_n=0$, is a bijection. If $k$ has characteristic $0$, then, moreover,
the map
$$
H^1_{\et}(A,G)\to H^1_{\et}(R,G)
$$
has trivial kernel.
\end{cor}

Theorem~\ref{thm:H1} and Corollary~\ref{cor:torsors} are proved in~\S~\ref{sec:Milnor} and~\S~\ref{sec:H1} respectively.

In parallel with the analog of the Bass--Quillen conjecture for discrete Hodge algebras,
T. Vorst~\cite[Theorem 1.1 (ii)]{Vo-SeHo} also established a similar result for the non-stable $SK_1$-functors $K_1^{\SL_n}(R)=\SL_n(R)/E_n(R)$, $n\ge 3$,
where $E_n(R)$ is the subgroup of $\SL_n(R)$ generated by the
elementary transvections $e+te_{ij}$, $1\le i\neq j\le n$, $t\in R$. Namely, one concludes that
if $K_1^{\SL_n}(R[x_1,\ldots,x_n])=K_1^{\SL_n}(R)$, then the same holds for any discrete Hodge algebra over $R$.
The corresponding analog of Serre's problem is supplied by~\cite{Sus}.
V.I. Kopeiko extended this observation to symplectic groups~\cite{Ko}.

More generally, for any reductive group $G$ over $R$ and a parabolic subgroup $P$ of $G$,
one defines the elementary subgroup $E_P(R)$ of $G(R)$ as the subgroup
generated by the $R$-points of unipotent radicals of parabolic subgroups of $G$, and considers the
corresponding non-stable $K_1$-functor $K_1^{G,P}(R)=G(R)/E_P(R)$~\cite{PS,St-poly}.
 In particular, if $A=k$ is a field, $E(k)$ is nothing but the group $G(k)^+$ introduced
by J. Tits~\cite{Tits64}, and $K_1^G(k)$ is the subject of the Kneser--Tits problem~\cite{Gil}. If $G$ has isotropic
rank $\ge 2$, then $K_1^{G,P}(R)$ is independent of $P$ and we denote it by $K_1^G(R)$, see \S~\ref{sec:K1} for
a formal definition.

We generalize the results of Vorst and Kopeiko as follows.
The proofs of Theorem~\ref{thm:square-free} and Corollaries~\ref{cor:square-free}--\ref{cor:split} are given in~\S~\ref{sec:K1}.

\begin{thm}\label{thm:square-free}
Let $G$ be a reductive
group scheme over a commutative Noetherian ring $R$, and let $P$ be a parabolic $R$-subgroup of $G$.
If $K_1^{G,P}(R[x_1,\ldots,x_n])=K_1^{G,P}(R)$ for any $n\ge 1$, then $K_1^{G,P}(A)=K_1^{G,P}(R)$
for any square-free discrete Hodge algebra $A$ over $R$.
\end{thm}

\begin{cor}\label{cor:square-free}
Let $G$ be a reductive
group scheme of isotropic rank $\ge 2$ over a field $k$.
Assume that either $R=k$, or $k$ is perfect and $R$ is a regular ring containing $k$.
Then for any square-free discrete Hodge algebra $A$ over $R$ one has $K_1^G(A)=K_1^G(R)$.
\end{cor}

For non-square-free discrete Hodge algebras the above result cannot be true, since, for example,
$\Gm_{,k}(k[x]/x^2)\neq\Gm_{,k}(k)$. However, if $G$ is a Chevalley--Demazure (i.e. split) reductive group scheme, we are able
to show that central subtori are, essentially, the only problem.

\begin{cor}\label{cor:split-sc}
Let $G$ be a split simply connected semisimple group scheme
over $\ZZ$ and let $B$ be a Borel subgroup of $G$.
For any commutative Noetherian ring $R$,
if $K_1^{G,B}(R[x_1,\ldots,x_n])=K_1^{G,B}(R)$ for any $n\ge 1$, then $K_1^{G,B}(A)=K_1^{G,B}(R)$
for any discrete Hodge algebra $A$ over $R$.
\end{cor}

\begin{cor}\label{cor:split}
Let $G$ be a split reductive group scheme over $\ZZ$,
such that every semisimple normal subgroup of $G$ has semisimple rank $\ge 2$.
Let $R$ be a regular ring containing a field $k$. If $G$ is simply connected semisimple or $k$ has characteristic $0$,
then for any discrete Hodge algebra $A$ over $R$ the natural sequence of group homomorphisms
$$
1\to\ker\bigl(\rad(G)(A)\to \rad(G)(R)\bigr)\to K_1^G(A)\to K_1^G(R)\to 1
$$
is exact.
\end{cor}

Assume that $G$ is defined over an infinite perfect field $k$, and let $R$ be a smooth $k$-algebra.
By~\cite[Theorem 4.1.3]{AHW} (see also~\cite{Mo-book} for the $\GL_n$ case) we know that
$H^1_{Nis}(R,G)=\Hom_{\mathcal{H}(k)}(\Spec(R),BG)$,
where $\mathcal{H}(k)$ is
 the Morel--Voevodsky $\Aff^1$-homotopy category  over $k$~\cite{MoV}. Combining~\cite[Theorem 4.3.1]{AHW}
with~\cite[Theorem 1.3]{St-poly}, one concludes that $K_1^G(R)=\Hom_{\mathcal{H}(k)}(\Spec(R),G)$. The
results of the present paper suggest that this relationship may somehow extend to non-smooth $k$-algebras.
See also Remark~\ref{rem:long} in~\S~\ref{sec:K1}.

\section{Discrete Hodge algebras as pull-backs}

Following Vorst~\cite{Vo-GLHo}, for any square-free discrete Hodge algebra $A$ over $R$ denote by $m_0(A)$ the minimal integer
$m$ such that $A\cong (R[x_1,\ldots,x_m]/I)[x_{m+1},\ldots,x_n]$, where $I$ is generated by square free monomials.
Note that there is a natural bijective correspondence between simplicial subcomplexes $\Sigma$
of a standard n-simplex $\Delta_n$
and square-free discrete Hodge algebras which are quotients of $R[x_0,\ldots,x_n]$ by the ideal generated by all monomials
that do not occur as faces of $\Sigma$~\cite[3.3]{Vo-SeHo}. This yields an easy geometric proof of the following
statement.

\begin{lem}\cite[3.4]{Vo-SeHo}\label{lem:ind}
Let $A$ be a square-free discrete Hodge algebra over $R$ with $m_0(A)>0$, then there exist square-free discrete Hodge algebras $A_1$ and
$A_2$ over $R$ and a Cartesian square of rings
\begin{equation}\label{eq:ind-square}
\xymatrix@R=20pt@C=35pt{
A\ar[d]^{i_2}\ar[r]^{i_1}&A_1\ar[d]^{j_1}\\
A_2[x]\ar[r]^{j_2}&A_2\\
}
\end{equation}
such that all maps are surjective, $j_2$ is the evaluation at $x=0$, and $m_0(A_2)<m_0(A)$, $m_0(A_1)<m_0(A)$.
\end{lem}

The following lemma is a slightly modified version of~\cite[Lemma 5.7, attributed to C. Weibel and R. G. Swan]{And}.

\begin{lem}\label{lem:graded-bij}
Let $R$ be a commutative ring with 1, and let $F$ be a covariant functor on the category of commutative
finitely generated $R$-algebras
with values in pointed sets. Let $A=\bigoplus_{i\ge 0}A_i$ be a graded $R$-algebra such that the map $F(A[x])\to F(A)$
induced by the evaluation $A[x]\xrightarrow{x\mapsto 0}A$ has trivial kernel (respectively, is bijective).
Then the map $F(A)\to F(A/\bigoplus_{i\ge 1}A_i)$ induced by the quotient homomorphism
has trivial kernel (respectively, is bijective).
\end{lem}
\begin{proof}
Define $f:A=\bigoplus_{i\ge 0}A_i\to A[x]$ by $f(\sum_{i=0}^na_i)=\sum_{i=0}^na_ix^i$. Let $t_a:A[x]\to A[x]$ denote the
automorphism induced by $x\mapsto x+a$, $a\in A$, and let $e_a:A[x]\to A$ denote the evaluation at $x=a$.
Since $e_1\circ f=\id_A$, we conclude that $F(f)$ is injective. By assumption $F(e_0)$ has trivial kernel,
hence $F(e_0\circ f)=F(e_0)\circ F(f)$
also has trivial kernel. But the latter map factors through the quotient
homomorphism $A\to A/\bigoplus_{i\ge 1}A_i$.

Similarly, if $F(e_0)$ is bijective, then $F(e_1)=F(e_0\circ t_1)$ is also bijective, and hence $F(f)$ is bijective. Then
$F(e_0\circ f):F(A)\to F(A/\bigoplus_{i\ge 1}A_i)\cong F(A_0)$ is bijective.
\end{proof}

\begin{cor}\label{cor:H-graded}
Let $R$ be a commutative ring with 1, and let $F$ be a covariant functor on the category of commutative
finitely generated $R$-algebras with values in pointed sets. For any discrete Hodge algebra $A=R[x_1,\ldots,x_n]/I$ over $R$,
if $F(A[x])\xrightarrow{x\mapsto 0}F(A)$ has trivial kernel (respectively, is bijective), then
the canonical projection $F(A)\xrightarrow{x_i\mapsto 0} F(R)$ has the same property.
\end{cor}
\begin{proof}
Lemma~\ref{lem:graded-bij} applies, since the algebra $A$ inherits the total degree grading from $R[x_1,\ldots,x_n]$.
\end{proof}

\begin{lem}\label{lem:dh-functor}
Let $R$ be a commutative ring with 1. Let $F$ be a covariant functor on the category of commutative finitely generated $R$-algebras
with values in pointed sets.

(i) Assume that for any
Cartesian square of commutative finitely generated $R$-algebras
\begin{equation}\label{eq:sq-F}
\xymatrix@R=20pt@C=35pt{
A\ar[d]^{i_2}\ar[r]^{i_1}&A_1\ar[d]^{j_1}\\
A_2[x]\ar[r]^{j_2}&A_2\\
}
\end{equation}
where $j_1$ is surjective and $j_2$ is the evaluation at $x=0$, the map of sets
\begin{equation}\label{eq:map-inj}
F(A)\xrightarrow{(i_2,i_1)} F(A_2[x])\times F(A_1)
\end{equation}
has trivial kernel. Assume also that the maps $g_m:F(R[x_1,\ldots,x_m])\to F(R)$ induced by evaluation at
$x_1=x_2=\ldots=x_m=0$ have trivial kernel for any $m\ge 1$. Then for any
square-free discrete Hodge algebra $A=R[x_1,\ldots,x_n]/I$ over $R$
the map $F(A)\to F(R)$ induced by evaluation at $x_1=x_2=\ldots=x_n=0$
has trivial kernel.

(ii) Assume instead that for any square~\eqref{eq:sq-F}, whenever $F(j_2)$ bijective, $F(i_1)$ is also bijective,
and that all maps $g_m$, $m\ge 1$, are bijective. Then
for any square-free discrete Hodge algebra $A$ over $R$ the map $F(A)\to F(R)$ is bijective.

\end{lem}
\begin{proof}

We apply induction on $m_0(A)$. If $m_0(A)=0$, then
$A=R[x_1,\ldots,x_n]$, and the claims hold. Assume that the claim holds for
any square-free discrete Hodge algebra $C$ over $R$ with $m_0(C)<m_0(A)$. By Lemma~\ref{lem:ind} there is a Cartesian
square~\eqref{eq:ind-square}. Assume (i). The induced square
\begin{equation}\label{eq:ind-square}
\xymatrix@R=20pt@C=35pt{
A[y]\ar[d]^{i_2}\ar[r]^{i_1}&A_1[y]\ar[d]^{j_1}\\
A_2[x,y]\ar[r]^{j_2}&A_2[y]\\
}
\end{equation}
is also Cartesian, and hence $F(A[y])\xrightarrow{(i_1,i_2)} F(A_1[y])\times F(A_2[x,y])$ has trivial kernel.
By the induction hypothesis the maps $F(A_1[y])\to F(A_1)$ and $F(A_2[x,y])\to F(A_2[x])$ induced by evaluation at $y=0$
have trivial kernel, hence $F(A[y])\to F(A)$ has trivial kernel.
By Corollary~\ref{cor:H-graded} the map $F(A)\to F(R)$ has trivial kernel.

In (ii), similarly, the map $F(A_2[x])\to F(A_2)$ is bijective by
induction assumption, hence $F(A)\to F(A_1)$ is bijective.
Again by induction assumption, $F(R)\to F(A_1)$ is bijective,
hence $F(R)\to F(A)$ is bijective.
\end{proof}

\begin{rem}\label{rem:nilp}
Let $A=R[x_1,\ldots,x_n]/I$, where $I$ is any ideal generated by monomials, be any discrete Hodge algebra over $R$.
Let $I_0\subseteq R[x_1,\ldots,x_n]$
be the ideal generated by monomials $x_1^{1-\delta_{i_1,0}}x_2^{1-\delta_{i_2,0}}\!\dots x_n^{1-\delta_{i_n,0}}$ for all
monomials $x_1^{i_1}x_2^{i_2}\dots x_n^{i_n}$ generating $I$; here $\delta_{i,0}$ denotes Kronecker delta.
Then $R[x_1,\ldots,x_n]/I_0$ is a square-free discrete Hodge algebra, and the kernel of
$$
\rho:A\to R[x_1,\ldots,x_n]/I_0
$$
is the nilpotent ideal $J=I_0/I$. Thus, if the functor $F$ of Lemma~\ref{lem:dh-functor} is such that $F(A)\to F(A/J)$
has trivial kernel, or, respectively, is bijective, then the claim (i), or, respectively, (ii) of the lemma holds for
any discrete Hodge algebra.
\end{rem}

\section{Milnor squares and $G$-torsors}\label{sec:Milnor}

Recall that
a Cartesian square of rings
\begin{equation}\label{eq:milnor-sq}
\xymatrix@R=20pt@C=35pt{
A\ar[d]^{i_1}\ar[r]^{i_2}&A_2\ar[d]^{j_2}\\
A_1\ar[r]^{j_1}&B\\
}
\end{equation}
is called a Milnor square, if at least one of the maps $j_1$, $j_2$ is surjective~\cite[\S 2]{Milnor}.
Note that if, say,
$j_1$ is surjective (respectively, split surjective), then $i_2$ has the same property. J. Milnor showed that these squares
have patching property for finitely generated projective modules. We need the following extension of this result.

\begin{lem}\label{lem:milnor-patch}
Let $G$ be a faithfully flat affine group scheme locally of finite presentation over a commutative ring $R$. Consider a
Milnor square of $R$-algebras~\eqref{eq:milnor-sq}, where $j_1$ is surjective.

(i) For any fppf $G$-torsors $E_1$ and $E_2$ over $A_1$ and $A_2$ respectively, and any
$G$-equivariant isomorphism $\phi:j_1^*(E_1)\to j_2^*(E_2)$ over $B$, there is a $G$-torsor
$E=E_1\cup_\phi E_2$ over $A$ and $G$-equivariant isomorphisms $\psi_1:i_1^*(E)\to E_1$, $\psi_2:i_2^*(E)\to E_2$
compatible with $\phi$.

(ii) For any fppf $G$-torsor $E$ over $A$, there is a natural isomorphism
$i_1^*(E)\cup_{\id} i_2^*(E)\xrightarrow{\cong} E$ of $G$-torsors over $A$.
\end{lem}
\begin{proof}
Since $G$ is affine, $E_1$ and $E_2$
 are affine over the respective bases, see e.g~\cite[\S 6.4]{Neron-book}. Let $G=\Spec(T)$,
 $E_1=\Spec(S_1)$, $E_2=\Spec(S_2)$.
Then $j_1^*(E_1)=\Spec(S_1\otimes_{A_1} B)$ and $j_2^*(E_2)=\Spec(S_2\otimes_{A_2} B)$ are isomorphic $B$-algebras.
We define $E$ to be the spectrum of the fibered product $S$ of rings $S_1$ and $S_2$ over
 $S_1\otimes_{A_1} B$, where the homomorphism $S_2\to S_1\otimes_{A_1} B$ factors through
$\phi:S_2\otimes_{A_2} B\to S_1\otimes_{A_1} B$. In other words, $E$ is the push-out of the diagram
\begin{equation}
\xymatrix@R=20pt@C=35pt{
j_1^*(E_1)\ar[d]^{(j_1)_{E_1}}\ar[r]^{\phi}&j_2^*(E_2)\ar[r]^{(j_2)_{E_2}}&E_2\\
E_1&&\\
}
\end{equation}
in the category of ringed spaces, see e.g.~\cite[Th. 5.1]{Ferrand}.
Note that the closed embedding $(j_1)_{E_1}$ base-changes to a closed embedding $E_2\to E$, and $j_1^*(E_1)\cong E_1\times_E E_2$.

By the universal property of the push-out $E$ is naturally
an $A$-scheme.
Clearly, considered as an $A$-module, $S$ is the Milnor-type patching of the flat $A_1$-module $S_1$ and
the flat $A_2$-module $S_2$ in the sense of~\cite[Th. 2.2]{Ferrand}. In particular, $S$ is flat over $A$
and $S\otimes_A A_i\cong S_i$, $i=1,2$.
Since $T$ is a faithfully flat $R$-algebra, $G_A\times_A E$ is faithfully flat over $E$,
hence
$$
G_A\times_A E\cong (G_{A_1}\times_{A_1} E_1)\cup_{id\times \phi}(G_{A_2}\times_{A_2} E_2)
$$
is the push-out of $G_{A_1}\times_{A_1} E_1$ and $G_{A_2}\times_{A_2} E_2$.
The universal property of push-out together with the $G$-equivariance of $\phi$ then defines an action of $G_A$ on $E$,
compatible with the actions of $G_{A_1}$ on $E_1$ and $G_{A_2}$ on $E_2$.

As a topological space, $E$ is isomorphic to the union of images of $E_1$ and $E_2$~\cite[Scolie 4.3]{Ferrand}, hence
$E\to\Spec(A)$ is surjective, and $S$ is faithfully flat over $A$. Then tensoring with $S$ also preserves fibered products
of $R$-algebras, hence
$$
E\times_A E\cong (E_{A_1}\times_{A_1} E_1)\cup_{id\times \phi}(E_{A_2}\times_{A_2} E_2)\cong
(E_1\times_{A_1} E_1)\cup_{\phi\times \phi}(E_2\times_{A_2} E_2).
$$
Since $G_{A_i}\times_{A_i} E_i$ is isomorphic to $E_i\times_{A_i}E_i$, $i=1,2$, by means of the map $(g,x)\mapsto (gx,x)$,
we conclude that
$$
E\times_A E\cong (G_{A_1}\times_{A_1} E_1)\cup_{\id\times \phi}(G_{A_2}\times_{A_2} E_2)\cong G_A\times_A E.
$$
Since $E\to\Spec(A)$ is a faithfully flat and quasi-compact morphism, and $G$ is locally of finite presentation, we conclude that
$E\to\Spec(A)$ is also locally of finite presentation by~\cite[Proposition 2.7.1]{EGAIV-2}.
Hence $E$ is an fppf $G$-torsor over $A$.

To prove the last claim of the lemma, let $E$ be any torsor over $A$. Note that both
$E$ and $i_1^*(E)\cup_{\id} i_2^*(E)$ are affine and flat over $A$, and
there is a morphism of $A$-schemes $i_1^*(E)\cup_{\id} i_2^*(E)\to E$.
By~\cite[Th. 2.2 (iv)]{Ferrand} patching of flat modules is an equivalence of categories, hence
it is an isomorphism.
Its $G$-equivariance is clear, since $G_A\times_A (i_1^*(E)\cup_{\id} i_2^*(E))$ is the patching
of $G_{A_1}\times_{A_1} i_1^*(E)$ and $G_{A_2}\times_{A_2} i_2^*(E)$.
\end{proof}

\begin{lem}\label{lem:ext}
In the setting of Lemma~\ref{lem:milnor-patch},
assume that $j_1$ has a section $s:B\to A_1$  which is a homomorphism of $R$-algebras, and let $r:A_2\to A$ be the induced section of $i_2$.
If $E_1$ is extended from $B$, then $E=E_1\cup_\phi E_2$ is extended from $A_2$, i.e. $E_1\cong s^*(j_1^*(E_1))$ implies
$E\cong r^*(E_2)$. In particular, $E$ is a trivial $G$-torsor if and only if $E_1$ and $E_2$ are trivial $G$-torsors.
\end{lem}
\begin{proof}
Since $E_2$ is faithfully flat over $A_2$, $r^*(E_2)=E_2\times_{A_2}A$ is faithfully flat over $A$.
By Lemma~\ref{lem:milnor-patch} (ii) the $G$-torsor
$r^*(E_2)$ over $A$ is isomorphic to the push-out of the $G$-torsors
$i_1^*(r^*(E_2))$ over $A_1$ and $i_2^*(r^*(E_2))=(i_2\circ r)^*(E_2)=\id_{A_2}^*(E_2)=E_2$
over $A_2$, by means of the trivial isomorphism of their restrictions to $B$. One has
$$
i_1^*(r^*(E_2))=(i_1\circ r)^*(E_2)=(s\circ j_2)^*(E_2)=s^*(j_2^*(E_2)).
$$
Then the isomorphism $\phi:j_1^*(E_1)\to j_2^*(E_2)$ extends by means of $s^*$ to the isomorphism  of $A_1$-torsors
$E_1\cong s^*(j_1^*(E_1))\to i_1^*(r^*(E_2))$. Hence $r^*(E_2)\cong E_1\cup_\phi E_2$ by the unicity of the push-out.
\end{proof}

\begin{proof}[Proof of Theorem~\ref{thm:H1}]
(i) We check the conditions of Lemma~\ref{lem:dh-functor} for the functor $H^1_\tau(-,G)$. For any square~\eqref{eq:sq-F},
Lemma~\ref{lem:ext} readily implies that that a
$G$-torsor $E$ over $A$ is trivial, once $i^*_2(E)$ and $i_1^*(E)$ are trivial.
In the bijective case, for any $E$ over $A$, we know that $i^*_2(E)$ is extended from $A_2$.
Then $E$ is extended from $A_1$ by Lemma~\ref{lem:ext}.

(ii) If $G$ is smooth, then for any commutative $R$-algebra $A$ and for any nilpotent ideal $J$ of $A$, the map
$H^1_{fppf}(A,G)\to H^1_{fppf}(A/J,G)$ is
bijective~\cite{Gr-BrauerIII,Strano-henspairs}. Then by Remark~\ref{rem:nilp} for any discrete Hodge algebra
$A=R[x_1,\ldots,x_n]/I$ over $R$
one has $H^1_{fppf}(A,G)\cong H^1_{fppf}(R[x_1,\ldots,x_n]/I_0,G)$, where $A'=R[x_1,\ldots,x_n]/I_0$
is a square-free discrete Hodge algebra. Then, clearly, $H^1_\tau(A,G)\to H^1_\tau(A',G)$ is injective.
If $H^1_\tau(A',G)\to H^1_\tau(R,G)$ has trivial kernel, this implies that $H^1_\tau(A,G)\to H^1_\tau(R,G)$ has trivial kernel.
In the bijective case, we conclude that $H^1_\tau(A,G)\to H^1_\tau(R,G)$ is injective, and its surjectivity
is automatic since $A\to R$ has a section.
\end{proof}

\section{Torsors under reductive group schemes}\label{sec:H1}

In the present section we apply the results of~\S~\ref{sec:Milnor} to isotropic reductive groups and prove
Corollary~\ref{cor:torsors}.

\begin{lem}\label{lem:TG'GT}
Let $R$ be a regular semilocal domain, $K$ be the fraction field of $R$. Let $G,G'$ be reductive $R$-groups, and $T$ be an
$R$-group of multiplicative type such that there is a short exact sequence
$$
(a)\quad 1\to G'\to G\to T\to 1\quad\mbox{or}\quad (b)\quad 1\to T\to G'\to G\to 1
$$
of $R$-group schemes.
For any $n\ge 1$, if the natural maps
$$
H^1_{\et}(R[x_1,\ldots,x_n],G')\to H^1_{\et}(K[x_1,\ldots,x_n],G')
$$
and
$H^1_{\et}(R,G)\to H^1_{\et}(K,G)$ have trivial kernels, then
$$
H^1_{\et}(R[x_1,\ldots,x_n],G)\to H^1_{\et}(K[x_1,\ldots,x_n],G)
$$
has trivial
kernel.
\end{lem}
\begin{proof}
Since $G$ and $G'$ are smooth, we can replace their \'etale cohomology by fppf.
For shortness, write $\mathbf{x}$
instead of $x_1,\ldots,x_n$, and $\mathbf{x}=0$ instead of $x_1=\ldots=x_n=0$.

Consider first the case $(a)$. Let $S$ be any of $R$, $R[\mathbf{x}]$, $K$, $K[\mathbf{x}]$, then we have
an exact sequence of pointed sets
$$
T(S)\to H^1_{fppf}(S,G')\to
H^1_{fppf}(S,G)\to H^1_{fppf}(S,T).
$$
By~\cite[Lemma 2.4]{CTS} $H^1_{fppf}(S,T)\cong H^1_{fppf}(S[x],T)$,
and by~\cite[Theorem 4.1]{CTS} $H^1_{fppf}(R,T)\to H^1_{fppf}(K,T)$ is injective.
Hence $H^1_{fppf}(R[\mathbf{x}],T)\to H^1_{fppf}(K[\mathbf{x}],T)$ is also injective.
Hence any
$\xi\in\ker\bigl(H^1_{fppf}(R[\mathbf{x}],G)\to H^1_{fppf}(K[\mathbf{x}],G)\bigr)$ lifts to an element $\eta\in H^1_{fppf}(R[\mathbf{x}],G')$.
Let $\bar\eta$ be the image of $\eta$ in $H^1_{fppf}(G',K[\mathbf{x}])$. Then there is $\theta\in T(K[\mathbf{x}])$ such that
$\bar\eta$ is the image of $\theta$. Since $T$ is a group of multiplicative type, we have $T(K[\mathbf{x}])=T(K)$. Hence
$\bar\eta$ is extended from $K$. By the assumptions on $\xi$ and $G$, the torsor $\xi|_{\mathbf{x}=0}$ is trivial, hence $\eta|_{\mathbf{x}=0}$
has a preimage $\sigma\in T(R)$. Clearly, the image of $\sigma$ in $T(K)$ maps to $\bar\eta|_{\mathbf{x}=0}$.
Note that the group $T(R[\mathbf{x}])=T(R)$ acts on
$H^1_{fppf}(R[\mathbf{x}],G')$ by right shifts, see~\cite[\S 5.5]{Serre-gal}.
Since $\bar\eta$
is extended from $K$, the image of $\eta\cdot\sigma^{-1}$ in $H^1_{fppf}(K[\mathbf{x}],G')$ is trivial. Hence $\eta\cdot\sigma^{-1}$
is trivial. Hence $\eta$ is extended from $R$. Hence $\xi$ is extended from $R$. Then $\xi$ is trivial by the assumption on $G$.

Consider the case $(b)$. For each $S$ as above, we have
an exact sequence
$$
H^1_{fppf}(S,T)\to H^1_{fppf}(S,G')\to
H^1_{fppf}(S,G)\to H^2_{fppf}(S,T).
$$
By~\cite[Theorem 4.3]{CTS}
$H^2_{fppf}(R[\mathbf{x}],T)\to H^2_{fppf}(K(\mathbf{x}),T)$ is injective, hence
$H^2_{fppf}(R[\mathbf{x}],T)\to H^2_{fppf}(K[\mathbf{x}],T)$ is also injective.
The rest of the proof is the same as in the previous case, with the only difference that one uses the action
of the commutative group $H^1_{fppf}(R[\mathbf{x}],T)\cong H^1_{fppf}(R,T)$ on
$H^1_{fppf}(R[\mathbf{x}],G')$, which is well-defined since $T$ is central in $G'$; see~\cite[\S 5.7]{Serre-gal}.
\end{proof}

The following statement for simply connected semisimple reductive groups is a particular case of~\cite[Theorem 1.6]{PaStV}.
We use this case, as well as the result of I. Panin and R. Fedorov on the Serre--Grothendieck conjecture~\cite{PaF},
to obtain the case of general reductive groups.

\begin{thm}\label{thm:ker-poly}
Assume that $R$ is a regular semilocal domain that contains an infinite field, and let $K$ be its fraction field. Let
$G$ be a reductive group scheme over $R$
of isotropic rank $\ge 1$. Then for any $n\ge 1$ the natural map
$$
H^1_{\et}(R[x_1,\ldots,x_n],G)\to H^1_{\et}(K[x_1,\ldots,x_n],G)
$$
has trivial kernel.
\end{thm}
\begin{proof}
In~\cite[Theorem 1.6]{PaStV} the claim is proved under the assumption that $G$ is simple and simply connected.
The case where $G$ is an arbitrary simply connected reductive group follows immediately by the Faddeev-Shapiro lemma,
as in the proof of~\cite[Theorem 11.1]{PaStV}. Now let $G$ be an arbitrary reductive group, let
$\der(G)$  be its derived subgroup in the sense of~\cite{SGA3}, and let $G^{sc}$ be the simply connected
cover of $\der(G)$. Then
$G^{sc}$ and $\der(G)$ are semisimple reductive groups satisfying the same isotropy condition as $G$.
There are two short exact sequences of reductive $R$-groups $1\to\der(G)\to G\to\corad(G)\to 1$, and
$1\to C\to G^{sc}\to \der(G)\to 1$,
where $\corad(G)$ and $C$ are $R$-groups of multiplicative type~\cite[Exp. XXII]{SGA3}.
 Note that
for any reductive group $H$ over $R$, the map $H^1_{\et}(R,H)\to H^1_{\et}(K,H)$ has trivial kernel by the corresponding
case of the Serre--Grothendieck conjecture~\cite[Theorem 1]{PaF}.
Since reductive groups are smooth by definition, we can replace the \'etale topology by fppf.
Hence these two short exact sequences are subject to
Lemma~\ref{lem:TG'GT}.
\end{proof}

The following statement is a slight extension of~\cite[Corollary 1.7]{PaStV}.

\begin{thm}\label{thm:et-ker}
Assume that $R$ is a regular domain that contains a field of characteristic $0$. Let $G$ be a reductive group scheme over $R$
of isotropic rank $\ge 1$. Then for any $n\ge 1$ the map
$$
H^1_{\et}(R[x],G)\to H^1_{\et}(R,G),
$$
induced by evaluation at $x=0$, has trivial kernel.
\end{thm}
\begin{proof}
In~\cite[Corollary 1.7]{PaStV} the claim is established under the assumption that $G$ is simple and simply connected.
The proof for any reductive group is exactly the same using Theorem~\ref{thm:ker-poly} instead
of its simply connected case~\cite[Theorem 1.6]{PaStV}.
\end{proof}

\begin{thm}\label{thm:Nis-bij}
Let $G$ be a reductive
group scheme over
a regular domain $R$ containing an infinite field $k$. Assume that $G$ has isotropic rank $\ge 1$.
Then the map
$$
H^1_{Nis}(R[x],G)\xrightarrow{x\mapsto 0} H^1_{Nis}(R,G)
$$
is a bijection.
\end{thm}
\begin{proof}
We need to show that any Nisnevich $G$-torsor $E$ over $R[x]$ is extended from $R$. By~\cite[Corollary 3.2]{Thomason} $G$
is linear, hence by the local-global principle for torsors~\cite[Theorem 3.2.5]{AHW} (see also~\cite[Korollar 3.5.2]{Mo})
it is enough to prove the same claim for every maximal localization of $R$. Thus, we can assume that $R$ is regular local.
By Theorem~\ref{thm:ker-poly} $H^1_{\et}(R[x],G)\to H^1_{\et}(K[x],G)$ has trivial kernel.
By~\cite[Proposition 2.2]{CTO} $H^1_{\et}(K[x],G)\to H^1_{\et}(K(x),G)$ has trivial kernel.
Hence
$H^1_{Nis}(R[x],G)\to H^1_{Nis}(K(x),G)$ has trivial kernel. Since every Nisnevich torsor over $K(x)$ is trivial,
therefore, every Nisnevich torsor over $R[x]$ is trivial, and hence extended from $R$.
\end{proof}

\begin{proof}[Proof of Corollary~\ref{cor:torsors}]
The first statement follows from Theorem~\ref{thm:Nis-bij} and Theorem~\ref{thm:H1} (ii). The second statement follows
from Theorem~\ref{thm:et-ker} and Theorem~\ref{thm:H1} (ii).
\end{proof}

\section{Non-stable $K_1$-functors}\label{sec:K1}

Let $R$ be a commutative ring with 1. Let $G$ be an isotropic reductive group scheme over $R$, and
let $P$ be a parabolic subgroup of $G$ in the sense of~\cite{SGA3}.
Since the base $\Spec R$ is affine, the group $P$ has a Levi subgroup $L_P$~\cite[Exp.~XXVI Cor.~2.3]{SGA3}.
There is a unique parabolic subgroup $P^-$ in $G$ which is opposite to $P$ with respect to $L_P$,
that is $P^-\cap P=L_P$, cf.~\cite[Exp. XXVI Th. 4.3.2]{SGA3}.  We denote by $U_P$ and $U_{P^-}$ the unipotent
radicals of $P$ and $P^-$ respectively.

\begin{dfn}\label{defn:E_P}\cite{PS}
The \emph{elementary subgroup $E_P(R)$ corresponding to $P$} is the subgroup of $G(R)$
generated as an abstract group by $U_P(R)$ and $U_{P^-}(R)$. We denote by $K_1^{G,P}(R)=G(R)/E_P(R)$ the pointed set of
cosets $gE_P(R)$, $g\in G(R)$.
\end{dfn}

Note that if $L'_P$ is another Levi subgroup of $P$,
then $L'_P$ and $L_P$ are conjugate by an element $u\in U_P(R)$~\cite[Exp. XXVI Cor. 1.8]{SGA3}, hence
the group
$E_P(R)$ and the set $K_1^{G,P}(R)$ do not depend on the choice of a Levi subgroup or an opposite subgroup
$P^-$ (and so we do not include $P^-$ in the notation).

The following lemma generalizes~\cite[Theorem 2.1 (ii)]{Vo-SeHo}.

\begin{lem}\label{lem:milnor-sq-K1}
Let $G$ be a reductive
group scheme over a commutative ring $R$, and let $P$ be a proper parabolic subgroup of $G$.
For any Milnor square of $R$-algebras~\eqref{eq:milnor-sq}, where $j_1$ is surjective,
the induced map of sets
\begin{equation}\label{eq:map-K}
K_1^{G,P}(A)\xrightarrow{(i_1,i_2)} K_1^{G,P}(A_1)\times_{K_1^{G,P}(B)} K_1^{G,P}(A_2)
\end{equation}
is surjective.
If, moreover, $j_1$ is split surjective with $R$-algebra section map $s:B\to A_1$, and $j_2$ is surjective, then the induced square
\begin{equation}\label{eq:K-sq}
\xymatrix@R=20pt@C=35pt{
K_1^{G,P}(A)\ar[d]^{i_1}\ar[r]^{i_2}&K_1^{G,P}(A_2)\ar[d]^{j_2}\\
K_1^{G,P}(A_1)\ar[r]^{j_1}&K_1^{G,P}(B)\\
}
\end{equation}
is a Cartesian square of sets.
\end{lem}
\begin{proof}
To prove surjectivity of~\eqref{eq:map-K}, let $g_1\in G(A_1)$, $g_2\in G(A_2)$ be such that $j_1(g_1)\in j_2(g_2)E_P(B)$. Since $E_P(A_1)$ surjects onto $E_P(B)$,
adjusting $g_1$ we obtain $j_1(g_1)=j_2(g_2)$. Since $G$ is left exact, there is $g\in G(A)$ such that $i_1(g)=g_1$ and  $i_2(g)=g_2$.

Next, assume that
$j_1$ is split surjective and $j_2$ is surjective, then
all four homomorphisms of~\eqref{eq:milnor-sq} are surjective, and $i_2$ is also split.
Let $g_1,g_2\in G(A)$ be such that $i_1(g_1)\in i_1(g_2)E_P(A_1)$ and
$i_2(g_1)\in i_2(g_2)E_P(A_2)$. Then $i_1(g_2^{-1}g_1)\in E_P(A_1)$ and $i_2(g_2^{-1}g_1)\in E_P(A_2)$. Then $g=g_2^{-1}g_1$
satisfies $i_1(g)\in E_P(A_1)$ and $i_2(g)\in E_P(A_2)$. We are going to show that $g\in E_P(A)$.

Since $i_2$ is surjective, adjusting $g$ by an element of $E_P(A)$, we can assume that $i_2(g)=1$. Let $s:B\to A_1$
be a splitting of $j_1$. By~\cite[Lemma 4.1]{St-poly} one has
$$
G(A_1,\ker(j_1))\cap E_P(A_1)=E_P(A_1,\ker(j_1))=E_P(\ker(j_1))^{E_P(s(B))}.
$$
Since $j_1(i_1(g))=j_2(i_2(g))=1$, one has $i_1(g)\in E_P(\ker(j_1))^{E_P(s(B))}$. Since the square is Cartesian,
$\ker(j_1)\subseteq i_1(\ker(i_2))$. Therefore, we can lift any element of $E_P(\ker(j_1))$ to an element of
$E_P(\ker(i_2))$. Since, moreover, $j_2$ is surjective, $i_1(g)$ has a preimage in $E_P(A,\ker(i_2))^{E_P(A)}=E_P(A,\ker(i_2))$. Since the square is Cartesian
and $G$ is left exact, we conclude that $g\in E_P(A,\ker(i_2))$. This finishes the proof.

\end{proof}

\begin{proof}[Proof of Theorem~\ref{thm:square-free}]
The claim follows immediately from Lemma~\ref{lem:dh-functor} and Lemma~\ref{lem:milnor-sq-K1} (ii).
\end{proof}

If $P$ is a strictly proper parabolic subgroup and $G$ has isotropic rank $\ge 2$, then $K_1^{G,P}$ is group-valued
and independent of $P$.

\begin{dfn}
A parabolic subgroup $P$ in $G$ is called
\emph{strictly proper}, if it intersects properly every normal semisimple subgroup of $G$.
\end{dfn}

\begin{thm}\label{th:PS-normality}\cite[Lemma 12, Theorem 1]{PS}
Let $G$ be a reductive group over a commutative ring $R$, and let $A$ be a commutative $R$-algebra.
If for any maximal ideal $m$ of $R$ the isotropic rank of $G_{R_m}$ is $\ge 2$,
then the subgroup $E_P(A)$ of $G(A)$ is the same for any
strictly proper parabolic $A$-subgroup $P$ of $G_A$, and is normal in $G(A)$.
\end{thm}

\begin{dfn}
Let $G$ be a reductive group of isotropic rank $\ge 2$ over a commutative ring $R$.
For any strictly proper parabolic subgroup $P$ of $G$ over $R$, and any $R$-algebra $A$,
we call the subgroup $E(A)=E_P(A)$, where $P$ is a strictly proper parabolic subgroup of $G$,
the \emph{elementary subgroup}  of $G(A)$.
The functor $K_1^G$ on the category of commutative $R$-algebras $A$, given by $K_1^G(A)=G(A)/E(A)$,
is called the \emph{non-stable $K_1$-functor} associated to $G$.
\end{dfn}
The normality of the elementary subgroup implies that $K_1^G$ is a group-valued functor.

\begin{proof}[Proof of Corollary~\ref{cor:square-free}]
Under the given assumptions $K_1^G(R[x_1,\ldots,x_n])=K_1^G(R)$ for any $n\ge 1$ by~\cite[Theorem 1.2]{St-poly}
and~\cite[Theorem 1.3]{St-poly} respectively. Hence the claim follows from Theorem~\ref{thm:square-free}.
\end{proof}

\begin{proof}[Proof of Corollary~\ref{cor:split-sc}]
Let $A=R[x_1,\ldots,x_n]/I$, where $I$ is any ideal generated by monomials.
By Remark~\ref{rem:nilp} there is a square-free discrete Hodge algebra $A'=A/J$ over $R$, where $J$ is an ideal of $A$ generated by a finite set
of nilpotent elements.
Let $B^-$ be a Borel $R$-subgroup of $G$ opposite to $B$, let
$U_B$ and $U_{B^-}$ be their unipotent radicals, and let $T=B\cap B^-$
be a maximal torus of $G$.
By Theorem~\ref{thm:square-free} we know that $K_1^{G,B}(A')=K_1^{G,B}(R)$.
We need to show that the natural homomorphism $p:K_1^{G,B}(A)\to K_1^{G,B}(A')$ is injective.
Let $g\in G(A)$ be such that $p(g)\in E_B(A')$. Adjusting $g$ by an element of $E_B(A)$,
we can assume that $p(g)=1$.  Then $U_BTU_B^-\subseteq G$ is a principal open subscheme of $G$ corresponding to
a function $f\in R[G]$ and isomorphic to the direct product $U_B\times T\times U_{B^-}$,
see~\cite[Exp. XXVI, Remarque 4.3.6]{SGA3}
and~\cite[p.~9]{Ma}. Since $p(g)=1$, we have $f(p(g))\in (A/J)^\times$. Since $J$ is nilpotent,
it follows that $f(g)\in A^\times$, and hence $g\in U_B(A)T(A)U_{B^-}(A)$. Since $T$ is a split maximal torus and $G$
is simply connected semisimple, we have $T(A)\le E(A)$. Hence $g\in E(A)$, as required.
\end{proof}

\begin{lem}\label{lem:ker-K-seq}
Let $R$ be a Noetherian ring, and let
$J\subseteq R$ be an ideal such that $(R,J)$ is a Henselian pair.
Let $1\to T\to G'\to G\to 1$ be a short exact sequence of $R$-group schemes, where $T$ is a smooth $R$-group of multiplicative
type, and $G,G'$ be two reductive $R$-groups. Let $P'$ be a parabolic $R$-subgroup of $G'$, and let $P$ be its image in $G$.
There is a short exact sequence of groups
\begin{multline}\label{eq:ker-K-seq}
1\to \ker\bigl(T(R)\to T(R/J)\bigr)\to\ker\bigl(K_1^{G',P'}(R)\to K_1^{G',P'}(R/J)\bigr)\\
\to
\ker\bigl(K_1^{G,P}(R)\to K_1^{G,P}(R/J)\bigr)\to 1.
\end{multline}
\end{lem}
\begin{proof}
Let $S$ be one of $R$, $R/J$. We have a short exact sequence of groups
$$
1\to T(S)\to G'(S)\to G(S)\to H^1_{fppf}(S,T).
$$
Since the image of $E_{P'}(S)$ in $G(S)$ coincides with $E_{P}(S)$, there is an induced sequence
$$
1\to T(S)\to K_1^{G',P'}(S)\to K_1^{G,P}(S).
$$
Hence the maps in~\eqref{eq:ker-K-seq} are well-defined, and it is a complex.

To prove the exactness of~\eqref{eq:ker-K-seq} at the third term, let $g\in G'(R)$ be such that
its image in $G(R)$ belongs to $E_{P}(R)$. Adjusting $g$ by an element of $E_{P'}(R)$, we can assume
that $g$ is in the kernel of $G'(R)\to G(R)$. Then $g$ belongs to the image of $T(R)$.

To prove the exactness of~\eqref{eq:ker-K-seq} at the fourth term, assume that $g\in G(R)$
maps to an element of
$E_P(R/J)$. Since $E_{P}(R)\to E_{P}(R/J)$ is surjective, adjusting $g$ we can assume that $g$ maps to $1\in G(R/J)$.
Since $T$ is smooth, we can replace its fppf cohomology with \'etale, and $H^1_{\et}(R,T)\cong H^1_{\et}(R/J,T)$
by~\cite{Gr-BrauerIII} (see also~\cite{Strano-henspairs}).
Hence $g$ lifts to an element $h\in G'(R)$. By assumption the image of $h$
in $G'(R/J)$ lifts to $T(R/J)$.
Since $(R,J)$ is a Henselian pair, and $T$ is smooth, the map
$T(R)\to T(R/J)$ is surjective.  Adjusting $h$ by an element of $T(R)$, we obtain a new preimage $h'\in G'(R)$ of $g$
such that its image in $G(R/J)$ is trivial. Hence $h'\in E_{P'}(R)$ and $g\in E_P(R)$, as required.

\end{proof}

\begin{proof}[Proof of Corollary~\ref{cor:split}]
Since $R$ contains a field, it contains a perfect field $k_0$. Since $G$ is split, it is defined over $k_0$.
Then $K_1^G(R[x_1,\ldots,x_n])=K_1^G(R)$ for any $n\ge 1$ by~\cite[Theorem 1.3]{St-poly}.
Let $G^{sc}$ be the simply connected cover of the semisimple group $G^{ss}=G/\rad(G)$.
Then by Corollary~\ref{cor:split-sc} $K_1^{G^{sc}}(A)=K_1^{G^{sc}}(R)$. If $G=G^{sc}$, this finishes the proof.

To treat the general case, recall that by Remark~\ref{rem:nilp} there is a square-free discrete Hodge algebra
$A'=A/J$ over $R$, where $J$ is an ideal of $A$ generated by a finite set of nilpotent elements. Clearly,
$(A,J)$ is a Henselian pair. There is a short
exact sequence $1\to C\to G^{sc}\to G^{ad}\to 1$, where
$C$ is an $R$-group of multiplicative type.
By Theorem~\ref{thm:square-free} $K_1^{G^{sc}}(A/J)=K_1^{G^{sc}}(R)$,
and $K_1^{G^{ss}}(A/J)=K_1^{G^{ss}}(R)$. In particular, $K_1^{G^{sc}}(A)=K_1^{G^{sc}}(A/J)$.
Hence by Lemma~\ref{lem:ker-K-seq} $K_1^{G^{ss}}(A)\to K_1^{G^{ss}}(A/J)$
has trivial kernel. Since $K_1^{G^{ss}}(A)\to K_1^{G^{ss}}(R)$ is obviously surjective, we conclude that
$K_1^{G^{ss}}(A)=K_1^{G^{ss}}(A/J)=K_1^{G^{ss}}(R)$. Applying Lemma~\ref{lem:ker-K-seq} again, we obtain
$$
\ker\bigl(\rad(G)(A)\to \rad(G)(A/J)\bigr)=\ker\bigl(K_1^G(A)\to K_1^G(A/J)\bigr).
$$
It remains to note that by Theorem~\ref{thm:square-free} $K_1^G(A/J)=K_1^G(R)$ and $\rad(G)(A/J)=\rad(G)(R)$.
\end{proof}

\begin{rem}\label{rem:long}
Let $G$ be a smooth group scheme over a commutative Noetherian ring $R$. Let $H$ be a normal subgroup subfunctor of the
functor represented by $G$ on the category of commutative finitely generated $R$-algebras, such that
for any surjective homomorphism $\phi:A\to B$ of such algebras one has $H(B)\le \phi(G(A))$. (If $G$
is a reductive group of isotropic rank $\ge 2$, one can take $H=E$.)
For any Milnor square of $R$-algebras~\eqref{eq:milnor-sq} there is an exact sequence
of pointed sets
\begin{multline*}
G(A)/H(A)\xrightarrow{\st (i_1,i_2)} G(A_1)/H(A_1)\times G(A_2)/H(A_2)
\xrightarrow{\st \lambda} G(B)/H(B)\to\\
\xrightarrow{\delta} H^1_{\textrm{\'et}}(A,G)\xrightarrow{\st (i_1^*,i_2^*)}
H^1_{\textrm{\'et}}(A_1,G)\times H^1_{\textrm{\'et}}(A_2,G).
\end{multline*}
Here the map $\lambda$ is defined by $\lambda(g_1,g_2)=j_1(g_1){j_2(g_2)}^{-1}$, and  $\delta$ sends the class of $g\in G(B)$ to the
$G$-torsor $G_{A_1}\cup_{g}G_{A_2}$ over $A$ constructed in Lemma~\ref{lem:milnor-patch}. The map $\delta$ is well-defined,
since any $g\in H(B)\le G(B)$ lifts to $g'\in H(A_1)\le G(A_1)$, and
replacing $G_{A_1}$ by the isomorphic $G$-torsor $(g'^{-1})^*(G_{A_1})$, we conclude that $G_{A_1}\cup_{g}G_{A_2}$
is isomorphic to a trivial $G$-torsor. The exactness is straightforward.
\end{rem}

\renewcommand{\refname}{References}

\end{document}